\documentclass[11pt]{amsart}% [12pt] indica la dimensione del testo, {report} il tipo di documento
\title[Hodge Theory and Deformations of Affine Cones]{Hodge Theory and Deformations of Affine Cones of Subcanonical Projective Varieties}
\author{Carmelo Di Natale}
\address{King's College London\\
 Floor 5, Waterloo Bridge Wing, Franklin Wilkins Building,\\
   150 Stamford Street, London SE1 9NH, United Kingdom}
\email[C.~Di Natale]{carmelo.di\_natale@kcl.ac.uk}
\author{Enrico Fatighenti}
\address{Mathematics Institute\\
  Zeeman Building, University of Warwick\\
  Coventry CV4 7AL, United Kingdom}
\email[E.~Fatighenti]{E.Fatighenti@warwick.ac.uk}
\author{Domenico Fiorenza}
\address{Dipartimento di Matematica "Guido Castelnuovo" \\
  Universit\`a degli Studi di Roma "La Sapienza"\\
 Piazzale Aldo Moro, 2, 00185 Roma, Italy}
\email[D.~Fiorenza]{fiorenza@mat.uniroma1.it}
\usepackage[english]{babel}
\usepackage{latexsym}
\usepackage[utf8]{inputenc}
\usepackage{babel}
\usepackage{fancyhdr}

\usepackage{mathrsfs}
\usepackage{geometry}
\usepackage{textcomp}

\usepackage[all]{xy}

\usepackage{lmodern}
\usepackage{mathrsfs}
\usepackage[T1]{fontenc}
\usepackage{epsfig}
 \usepackage{amsmath,amsfonts,amssymb,amsthm}
\usepackage{graphics}
\usepackage{graphicx}
\usepackage{verbatim}

\usepackage{xcolor}

\newcommand{\C}{\mathbb{C}}

\newcommand{\Z}{\mathbb{Z}}

\newcommand{\PP}{\mathbb{P}}

\newcommand{\of}{\mathcal{O}}

\newcommand{\de}{\partial}

\newcommand{\W}{\bigwedge}
\newcommand{\T}{\Theta}

\DeclareMathOperator{\prim}{prim}

\DeclareMathOperator{\HH}{HH}

\DeclareMathOperator{\Spec}{Spec}

\DeclareMathOperator{\Ext}{Ext}

\DeclareMathOperator{\Ker}{Ker}
\DeclareMathOperator{\Cok}{coKer}
\DeclareMathOperator{\dd}{d}

\DeclareMathOperator{\ddim}{dim}
\newtheorem{thm}{Theorem}[section]
\newtheorem{cor}[thm]{Corollary}

\newtheorem{lem}[thm]{Lemma}

\theoremstyle{definition}

\newtheorem{rmk}[thm]{Remark}

\begin{document}

\begin{abstract}
We investigate the relation between the Hodge theory of a smooth subcanonical $n$-dimensional projective variety $X$ and the deformation theory of the affine cone $A_X$ over $X$. We start by identifying $H^{n-1,1}_{\mathrm{prim}}(X)$ as a distinguished graded component of the module of first order deformations of $A_X$, and later on we show how to identify the whole primitive cohomology of $X$ as a distinguished graded component of the Hochschild cohomology module of the punctured affine cone over $X$. In the particular case of a projective smooth hypersurface $X$ we recover Griffiths' isomorphism between the  primitive cohomology of $X$ and certain distinguished graded components of the Milnor algebra of a polynomial defining $X$. The main result of the article can be effectively exploited to compute Hodge numbers of smooth subcanonical  projective varieties. We provide a few example computation, as well a SINGULAR code, for Fano and Calabi-Yau threefolds.
\end{abstract}

\maketitle

\tableofcontents
 
 \section{Introduction}
Deformation Theory and Hodge Theory are known to be closely related. Examples of this friendship can be found, for example, in the theory of Variations of Hodge Structures, or in the Griffiths Residue Theory (see the original papers \cite{griffiths} or \cite{voisin2} for a more modern and detailed exposition). The latter identifies the Hodge structure of a smooth projective hypersurface with a subalgebra of the Milnor algebra of the hypersurface itself, an important deformation-theoretic invariant. More precisely, if 
\[
X=\{[x_0,\ldots,x_{n+1}]\,|\, f(x_0,\dots,x_{n+1})=0\}\subseteq \mathbb{P}_{\mathbb{C}}^{n+1}
\]
 is a smooth, degree $d$, $n$-dimensional projective hypersurface, Griffiths work establishes an isomorphism, given by a higher residue map
\[
H^{p-1, n+1-p}_{\prim}(X) \cong (\mathcal{M}_f)_{pd-n-2}.
\]
The subscript \emph{$\prim$}  in the above formula stands for primitive cohomology,  and the object on the right is the degree $pd-n-2$ component of the \emph{Milnor algebra} \[ \mathcal{M}_f:= \C[x_0, \ldots, x_{n+1}]/J_f, \] 
where $J_f=(\frac{\de f}{\de x_0}, \ldots,\frac{\de f}{\de x_{n+1}})$, is the Jacobian ideal of $f$. The Milnor algebra $\mathcal{M}_f$ contains deformation data for $X$, in the sense that its degree $d$  component can be identified with the space of first order embedded deformations of $X$ in $\mathbb{P}_{\mathbb{C}}^{n+1}$.\\
Let us tackle the general case: from now on fix $X\subseteq \mathbb{P}_{\mathbb{C}}^{N}$ to be a smooth complex projective variety of $\mathrm{dim}\,X=n$ and arbitrary codimension in some projective space and consider the space of first order infinitesimal deformations of the affine cone $A_X$ of $X$. This space of first order deformations is classically known in the literature as $T^1_{A_X}$ (see \cite{schlessinger} and \cite{sernesi}). It is a naturally graded vector space and in \cite{schlessinger} Schlessinger shows that its degree 0 component is identified with the space of first order embedded deformations of the variety $X$ in $\mathbb{P}_{\mathbb{C}}^{N}$. When $X$ is a hypersurface, this module is nothing but the Milnor algebra of $X$ (up to a shift of degree $d$), but in general it is a much richer and more complicated object.
\\
The superscript 1 in $T^1_{A_X}$ refers to the fact that this is indeed the degree 1 component of a graded module $T^\bullet_{A_X}$.
As we will see, there are several possible definitions for this graded module, but under reasonable assumptions -- for example, $X$ projectively normal -- we can take as definition
\[T^i_{A_X} := \Ext^i_{\of_{A_X}}(\Omega^1_{A_X}, \of_{A_X}).\]
Here, $T^1_{A_X}$ represents -- as already mentioned -- first order deformations of $A_X$, while $T^2_{A_X}$ and $T^0_{A_X}$ encode respectively obstructions and infinitesimal automorphisms. It is important to notice that in this case -- and more generally in the isolated-singularity case -- both $T^1_{A_X}$ and $T^2_{A_X}$ will be finite-dimensional.\\
Now let us assume that the embedding $X\hookrightarrow \mathbb{P}_{\mathbb{C}}^{N}$ is \emph{subcanonical}, i.e. that $\omega_X\cong \of_X(m)$ for some $m\in\Z$. Subcanonical varieties occupies a special case in literature: for a survey see for example \cite{depoi}, or \cite{arbarello}. Under this condition, we find an interesting generalization of Griffiths Residue Theorem. Our first results concern a deep relation between first order deformations, obstructions and automorphisms of $X$, and the Hodge theory of $X$:
\begin{thm} \label{Thm 1 Intr}
Let $X$ be a smooth complex subcanonical projective variety of dimension $n$, and let $m\in\mathbb{Z}$ be the integer such that $\omega_X\cong \of_X(m)$.
\begin{enumerate}
\item There is a natural isomorphism
\[(T^0_{A_X})_{m} \cong H^{n,0}_{\prim}(X). \]
\item If $H^1(X, \of_X(k))=0$ for every $k\in \mathbb{Z}$, then there is a natural isomorphism
\[(T^1_{A_X})_{m} \cong H^{n-1,1}_{\prim}(X).\]
\item If also $H^2(X, \of_X(k))=0$ for every $k\in \mathbb{Z}$, then there is a natural isomorphism
\[ (T^2_{A_X})_{m}\cong H^{n-2,1}_{\mathrm{prim}}(X). \]
%where $\lambda\colon H^{n-2,1}(X)\to H^{n-1,2}(X)$ is the Lefschetz isomorphism.
\end{enumerate}
\end{thm}
Notice that the apparently quite restrictive conditions $H^i(X, \of_X(k))=0$ for every $k\in\mathbb{Z}$ and for $i=1,2$ in the above theorem are actually satisfied by a large class of projective manifolds. Namely every arithmetically Cohen-Macaulay variety $X$ satisfies the vanishing condition for $i=1$ if $\ddim X\geq 2$ and satisfies also the condition for $i=2$ if $\ddim X\geq 3$.\\ 
The results stated in Theorem \ref{Thm 1 Intr} have a surprising generalisation in terms of derived categories. It is well known (see \cite{barannikov}) that, if $X$ is Calabi Yau, there is an identification between the space of infinitesimal deformations of the derived category $D(X)$ and the Hodge Theory of $X$, given by the celebrated Hochschild--Kostant--Rosenberg isomorphism (see \cite{caldararu}). \\ 
The story is indeed famous in the literature: in the spirit of the work of Kontsevich \cite{kontsevich}, one first identifies the Hochschild cohomology of $X$ with the cohomological algebra of polyvector fields 
\[ \HH^{\bullet}(X)= 
%\bigoplus_{k} \HH^{k}(X)=
 \bigoplus_k\left(\bigoplus _{p+q=k}\HH^{p,q}(X)\right)=\bigoplus_k\left(\bigoplus _{p+q=k}H^q(X, \W^p\T_X)\right)\] 
and then uses Serre duality and $\omega_X \cong \of_X$ to see how this precisely encodes the Hodge Theory of $X$. More intrinsically, this is the canonical isomorphism
\[ \HH^\bullet (X) \cong \HH_\bullet (X) [n], \] 
between the Hochschild cohomology and the shifted Hochshild homology of the Calabi-Yau manifold $X$. \\
%Note that in a somehow unusual way we decide to denote with $\T_X$ the holomorphic tangent sheaf, with the aim of not causing confusion with the $T^i$'s modules.
Of course the above isomorphism relies heavily on the triviality of the canonical sheaf $\omega_X \cong \of_X$. In the more general case of $\omega_X \cong \of_X(m)$ we will consider deformations of the derived category of the \emph{punctured affine cone} ${U_X}=A_X \smallsetminus \{0\}$. Indeed observe that Kontsevich's results still apply since ${U_X}$ is a smooth quasiprojective variety (see \cite{caldararu}). In particular, by using the $\mathbb{C}^*$-action which determines the grading, we will focus on the subspace of weight $m$ of the cohomology algebra of polyvector fields on ${U_X}$ and prove the following decomposition result.
\begin{thm} \label{Thm 2 Intr}
Let $X$ be a smooth complex subcanonical projective variety of dimension $n$, and let $m\in\mathbb{Z}$ be the integer such that $\omega_X\cong \of_X(m)$.
Then 
\[
\HH^{p,q}(U_X)_m \cong H^{n-p+1,q}_{\mathrm{prim}}(X)\oplus H^{n-q,p}_{\mathrm{prim}}(X).
\]

\end{thm}

Theorem \ref{Thm 2 Intr} enables us to reinterpret the (embedded) Hodge Theory of $X$ as a piece of the deformation theory of the derived category of punctured cone $A$. Moreover, we can recover Theorem \ref{Thm 1 Intr} as a special case of this result. As a matter of fact if $\mathrm{dim}(X)\geq 2$ and
%depth$_0A_X \geq 3$ -- which follows from the condition 
$H^1(X, \of_X(k))=0$ for every $k\in \mathbb{Z}$
%, whenever $\mathrm{dim}(X)\geq 2$ -- 
%we can ignore the contribution from the vertex and so 
then we get an isomorphism 
\[
(T^1_{A_X})_m \cong H^1({U_X}, \T_{U_X})_m=H^{n-1,1}_{\prim}(X), 
\]
where we used that the factor $H^{n,1}_{\prim}(X)$ vanishes by Hard Lefschetz. Similarly, one gets the description for $T^0_{A_X}$ and $T^2_{A_X}$.

\subsubsection*{Acknowledgements}
The original idea of using the $T^1$ came out after several discussion with Edoardo Sernesi, and we are deeply grateful to him. EF would like to thank his supervisor Miles Reid for constant encouragement and support, and also Stephen Coughlan, Alice Cuzzucoli, Lorenzo De Biase, Paolo Tripoli for useful discussions and suggestions, and Rosemary Taylor for proofreading. CDN and EF would like to thank also Luca Migliorini for his ideas and encouragements in the early stages. Part of this work was completed during the visit of DF and CDN at the University of Warwick for the ``2014-2015 Warwick EPSRC Symposium on Derived Categories and Applications'', and we would like  to thank the organizer of this event. Part of this work was also completed whilst EF was visiting Korea Institute for Advanced Studies (KIAS): the author would like to thank again Professor Miles Reid for inviting and all the staff and academic member of the institute.

\section{The Infinitesimal Deformation Module of an Affine Isolated Singularity}\label{sec.Ax} 
In this section we will see that -- under very mild assumptions -- various pieces of the Hodge Theory of a smooth projective variety can be described in terms of the classical deformation modules of the affine cone over it.
\subsection{A Quick Review in Deformation Theory}
Let $Y$ be a (reasonable) complex scheme and $R$ a local Artin $\mathbb C$-algebra; recall that an infinitesimal deformation of $Y$ over $R$ is is the datum of a pull-back diagram 
\[\xymatrix{
Y \ar[r]^{i}\ar[d] & \mathcal{Y}\ar[d]^{p}\\
\Spec(\C)\ar[r] & \Spec(R)
}\]

%\[
%\begindc{\commdiag}[200]
%
%
%\obj(1,5)[15]{$Y$}
%\obj(4,5)[35]{$\mathcal{Y}$}
%\obj(1,3)[13]{ $\Spec(\C)$}
%\obj(4,3)[33]{ $\Spec(R)$}
%
%\mor{15}{35}{$i$}
%\mor{15}{13}{}
%\mor{13}{33}{}
%\mor{35}{33}{$p$}
%\enddc
%\]
with $i$ a closed immersion and $p$ a flat and proper morphism. Equivalently, a deformation of $Y$ over $R$ can be viewed as a morphism of sheaves of $R$-algebras $\of_R\rightarrow\of_Y$ such that $\of_R$ is flat over $R$ and $\of_R\otimes_R\mathbb C\simeq \of_Y$. In more concrete terms, defining a deformation of a scheme $Y$ means defining a flat family of schemes over a fat point whose central fibre is isomorphic to $Y$. Of course there is an obvious notion of isomorphism of deformations, which we will not spell out here in detail. \\
Studying the Deformation Theory of a scheme $Y$ means indeed analysing the functor
\begin{eqnarray} \label{Def_Y}
\mathrm{Def}_Y:&\mathfrak{Art}_{\mathbb C}&\xrightarrow{\hspace*{2cm}}\mathfrak{Set} \nonumber \\
&R&\mapsto\frac{\left\{\text{deformations of }Y\text{ over }R\right\}}{\text{isomorphism}}.
\end{eqnarray}
which is known to satisfy Schlessinger's axiomatics of deformation functor (for more details see \cite{schlessinger2}). \\ 
It is now well-established that functor $\mathrm{Def}_Y$ is governed by a fundamental invariant of the scheme $Y$ -- originally envisioned by Grothendieck, then deeply studied by Andr\' e, Quillen, Deligne, Illusie and many others -- that is the cotangent complex $\mathbb L_{Y}$. The most concrete way to define the cotangent complex is the following: start with an affine  $\mathbb C$-scheme $\mathrm{Spec}(B)$ and consider a quasi-free resolution $Q^{\bullet}(B)$. Now apply the K\" ahler-differential functor to $Q^{\bullet}\left(B\right)$ and get the complex of $B$-modules $\Omega^1_{Q^{\bullet}(B)}$. Set the (absolute) cotangent complex of $\mathrm{Spec}(B)$ to be
\begin{equation} \label{ctgnt cmplx 1}
\mathbb L_{\mathrm{Spec}(B)}:=\Omega^1_{Q^{\bullet}(B)}
\end{equation}
and notice that formula \eqref{ctgnt cmplx 1} provides us with a perfectly well-defined notion in the derived category $D\left(\mathfrak{Mod}_B\right)$. As a matter of fact if we change the quasi-free resolution of $B$ we end up with the same object in the derived category. In more formal terms, we can say that the cotangent complex is the total left derived functor (in the sense of Quillen Homotopical Algebra) of the functor of K\" ahler differentials. The construction of the cotangent complex as given in formula \eqref{ctgnt cmplx 1} sheafifies, so we can associate to our scheme $Y$ its absolute cotangent complex $\mathbb L_{Y}$, which lives in the derived category $D(\mathfrak{QCoh}(Y))$. \\
In general the complex $\mathbb L_{Y}$ is made of terms lying in both positive and negative degrees and encodes all the invariants determining the functor $\mathrm{Def}_{Y}$ through its Ext groups; we recommend \cite{schlessingerlicht} for an introduction to the subject. More precisely we will consider the modules
%\begin{equation*}
%T^1_Y:=\mathrm{Ext}^1_{\of_Y}(\mathbb L_{Y},\of_Y)\qquad T^2_Y:=\mathrm{Ext}^2_{\of_Y}(\mathbb L_{Y},\of_Y)\qquad T^0_Y:=\mathrm{Ext}^0_{\of_Y}(\mathbb L_{Y},\of_Y)
%\end{equation*}
\begin{equation*}
T^i_Y:=\mathrm{Ext}^i_{\of_Y}(\mathbb L_{Y},\of_Y),
\end{equation*}
which, for $i=0,1,2$ are called  the automorphism module, the first-order deformations module, and the obstruction module associated to $Y$, respectively.  It is now well-understood that:
\begin{enumerate}
\item $T^0_Y$ encodes infinitesimal automorphisms of the scheme $Y$;
\item $T^1_Y$ parametrises first-order deformations of $Y$, i.e deformations over the ring of dual numbers $\frac{\mathbb C\left[t\right]}{t^2}$;
\item $T^2_Y$ contains obstructions, meaning that for all surjection $S\twoheadrightarrow R$ of local Artin $\mathbb{C}$-algebras the obstructions to lifting a $R$-deformation of $Y$ to a $S$-deformation live in $\mathrm{Ext}^{2}(\mathbb L_{Y,\of_Y})$.
\end{enumerate}
\begin{rmk} \label{derived}
The above statements tell us that the cotangent complex $\mathbb L_{Y}$ is a much richer object than the deformation functor $\mathrm{Def}_{Y}$: for example, as a set-valued functor, $\mathrm{Def}_{Y}$ is not able to see the infinitesimal automorphisms of $Y$, which are actually captured by $\mathbb L_{Y}$. Also, the higher Ext groups of $\mathbb L_{Y}$ describe some more subtle deformations of $Y$ usually called higher obstructions: these objects are extremely important in the new Derived Algebraic Geometry. See, e.g., \cite{schumi} for the relation between the cohomology of the wedge powers of the (co-)tangent complex of $Y$ and its extended derived deformations.
 However, in the rest of the paper we will only be interested in studying classical deformation invariants, therefore we will only care about the three above mentioned groups. \end{rmk}

From now on we will focus on the study of the deformation modules attached to an affine cone over a smooth projective variety. Throughout the rest of the paper fix $X$ to be a smooth projectively normal variety and call $A_X$ the affine cone over it, so that $A_X=\Spec(\mathfrak{a}_X)$, where
\[
\mathfrak{a}_X= \bigoplus_k H^0(X, \of_X(k)).
\]
We will also suppose everywhere that the projective embedding of $X$ is subcanonical, so that $\omega_X \cong \of_X(m)$ for some integer $m$. Key references for most of the considerations below are \cite{schlessinger}, \cite{wahl}. We stress that their results hold in the more general case of an affine isolated singularity, but here we only consider the case of affine cones over smooth projective varieties.\\ 
Consider the $T^1$-deformation module 
\[ T^1_{A_X}:=\Ext^1_{\of_{A_X}}(\mathbb L_{A_X}, \of_{A_X}),  \] 
which measures the first-order deformations of the affine cone $A_X$: in particular notice that it is equipped with a natural grading, induced by the fact that $\mathfrak{a}_X$ is a graded algebra itself. The cone $A_X$ may deform in several ways and each graded component of $T^1_{A_X}$ roughly speaking represents the degree of the polynomial we are adding in order to deform. In the case $X$ is a degree $d$ hypersurface with defining polynomial $f$ the graded module $T^1_{A_X}$ coincides, up to a shift of $-d$, with the Milnor Algebra $\mathcal{M}_f$, i.e. we have
\[
T^1_{A_X}[-d]\cong \mathcal{M}_f
\]
as graded modules.\\
Now, not all deformations of $A_X$ lead to another affine cone over a projective variety. The easiest example is the case of $xy=0$ in $\C^2$ that is the cone over the points $[1,0],[0,1] \in \PP^1$. In this case, up to isomorphism the only possible first-order deformation of the cone is given by $xy+\varepsilon=0$, and any element of this family is not a cone over a projective variety except for $\varepsilon=0$. In fact this is straightforward to verify: just notice that the Milnor algebra is 
\[ (\mathcal{M}_f)_{m}\cong (\C[x,y]/(x,y))_m \cong \begin{cases}\C \text { if $m=0$}\\
0\,\text{ if $m\neq 0$}
\end{cases} , \]
hence $(T^1_{A_X})_0=(\mathcal{M}_f)_{2}=0$. 
A very natural question arises: what are the deformations of $A_X$ that lead to family of affine cones over a projective variety? The naive answer is that the polynomials that we add in order to deform must be homogeneous of the same degree as the (homogeneous) equations of $A_X$. Luckily, this is also the correct one. \\
So far, all of what we said is very classical: under the above interpretation, the degree 0 piece of the deformation module of the affine cone over the projective variety $X$ represents the embedded first-order deformations of $X$ inside $\PP^{N}$.
Somehow more precisely, we have an exact sequence 
\[ \ldots \rightarrow H^1(X, \of_X) \rightarrow (T^1_{A_X})_0 \rightarrow H^1(X, \Theta_X) \rightarrow H^2(X, \of_X) \rightarrow \ldots \]
Notice that if the two side terms are both zero, we have an isomorphism between $ (T^1_{A_X})_0$ and $H^1(X, \Theta_X)$, which allows us to identify $(T^1_{A_X})_0$ with all infinitesimal deformations of $X$. This is for example the case of a smooth Calabi-Yau of dimension $\geq 3$; by definition we have $H^1(X, \of_X)=H^2(X, \of_X)=0$. This coincides with the standard fact that all Calabi-Yau from dimension 3 onwards are projective, while for example in the K3 case we have a 19-dimensional algebraic family inside a 20-dimensional deformation space, recorded by the fact $H^{2}(X, \of_X) \cong H^{0}(X, \omega_X) \cong \C$. In general when either $H^{0,1}(X)$ or $H^{0,2}(X)$ are non zero  there is generally a difference between embedded deformations and non-embedded ones.\\

\subsection{$T^1$ and Hodge Theory}
Now we want to explore deeper the relation between $T^1_{A_X}$, the smooth projective variety\footnote{Actually, everything in what follows holds more generally for quasi-smooth varieties in weighted projective spaces.} $X$, and the punctured cone $U_X:=A_X \setminus\{ 0\}$. 
\begin{lem}\label{lemma.first}
For every $k\in \mathbb{Z}$, the relative tangent sheaf exact sequence
\[
0 \rightarrow \Theta_{U_X/X} \rightarrow \Theta_{U_X} \stackrel{\dd\pi}{\rightarrow} \pi^*(\Theta_X) \rightarrow 0
\]
induces a long exact sequence 
\[
\ldots \rightarrow H^1(X, \of_X(k)) \rightarrow H^1({U_X}, \T_{U_X})_{k} \rightarrow H^1(X, \Theta_X(k)) \stackrel{\lambda}{\rightarrow} H^2(X, \of_X(k)) \rightarrow \ldots .\]
where the maps $\lambda$ are the Lefschetz operators.
\end{lem}
\begin{proof}
The Euler vector field gives a trivialization $\Theta_{U_X/X}\cong \of_{U_X}$, see  \cite{badescu}. We therefore get the short exact sequence \[ 0 \rightarrow \of_{U_X} \rightarrow \Theta_{U_X} \rightarrow \pi^*(\Theta_X) \rightarrow 0,\] and so, passing to cohomology, the long exact sequence \[ \ldots \rightarrow \bigoplus_{k \in \Z} H^1(X, \of_X(k)) \xrightarrow{\lambda} \bigoplus_{k \in \Z}H^1({U_X}, \T_{U_X})_k \rightarrow \bigoplus_{k \in \Z} H^1(X, \Theta_X(k)) \rightarrow \bigoplus_{k \in \Z} H^2(X, \of_X(k)) \rightarrow \ldots \] where the grading on $H^1({U_X}, \T_{U_X})$ is induced by the $\C^*$-action, and the connecting homomorphism $\lambda$ is the cup product with the extension class \[\Lambda:=\left[0\rightarrow\of_{U_X}\rightarrow\T_{U_X}\rightarrow\pi^*\T_X\rightarrow 0\right],\] which is an element in
 \[\mathrm{Ext}_{U_X}^1(\pi^*\T_X,\of_{U_X})\cong H^1(U_X,\pi^*\Omega^1_X)\cong\underset{s}{\bigoplus}H^1(X,\Omega^1_X(s)),\]
 see \cite{schlessinger}([Lemma 1, page 158]) and \cite{sernesi}. Notice that the map $\lambda$ is not a priori a morphism of graded modules: we should expect it to have several homogeneous components
 \begin{equation*}
\lambda_s: H^i(X,\T_X(k))\longrightarrow H^{i+1}(X,\of_X(k+s)),
\end{equation*}
which are identified with cohomology classes in $H^1(X,\Omega^1(s))$.
So our next step consists in showing that actually $\lambda$ reduces to its degree zero component $\lambda_0$, i.e., that $\Lambda$ consists into a single cohomology class in $H^1(X,\Omega^1)$. To see this, recall from  \cite{atiyah} (see \cite{contact} for a more modern treatment) that given a line bundle $L$ on a smooth complex manifold $X$, if we denote by $L^{\circ}$ the total space of the dual bundle $L^*$ with the zero section removed, then we have a canonical
 short exact sequence of sheaves of $\mathcal{O}_{L^{\circ}}$-modules
 \[
  0 \to \pi^*\Omega^1_X \to \Omega^1_{L^{\circ}} \to \mathcal{O}_{L^{\circ}} \to 0.
 \]
 Pushing forward to $X$ we get for every $k\in \mathbb{Z}$ a short exact sequence
 \[ 0 \to \Omega^1_X(k) \to \mathcal{L}_k \to L^{\otimes k}\to 0,\]
 where $\mathcal{L}_k$ denotes the degree $k$ component of $\pi_*\Omega^1_{L^{\circ}}$. In particular, if $L=\mathcal{O}_X(1)$, so that $L^{\circ}\cong U_X$, we get the short exact sequences
 \[ 0 \to \Omega^1_X(k) \to (\pi_*\Omega^1_{U_X})_k \to \of_X(k)\to 0,\]
 and the projection formula together with the isomorphisms $\pi^*\of_X\cong \of_{U_X}\cong \pi^*\of_X(k)$, shows that these are indeed all obtained from the single short exact sequence
 \[
 0 \to \Omega^1_X \to (\pi_*\Omega^1_{U_X})_0 \to \of_X\to 0
 \] 
by tensoring it by $\of_X(k)$. This implies that the extension class $[0 \to \Omega^1_X(k) \to (\pi_*\Omega^1_{U_X})_k \to \of_X(k)\to 0]$ is actually independent of $k$, and so the total extension class $[0 \to \pi^*\Omega^1_X \to \Omega^1_{U_X} \to \mathcal{O}_{U_X} \to 0]$ reduces to the extension class $[0 \to \Omega^1_X \to (\pi_*\Omega^1_{U_X})_0 \to \of_X\to 0]$, which is an element in $\mathrm{Ext}_{X}^1(\T_X,\of_{X})\cong H^1(X,\Omega_X^1)$, see \cite{bredon}. Since the short exact sequence $0 \to \pi^*\Omega^1_X \to \Omega^1_{U_X} \to \mathcal{O}_{U_X} \to 0$ is the dual of the short exact sequence $0 \rightarrow \of_{U_X} \rightarrow \Theta_{U_X} \rightarrow \pi^*(\Theta_X) \rightarrow 0$ we finally see that the connecting homomorphism $\lambda$ is indeed of degree zero and is given by the cup product with a distinguished element $\Lambda$ in $H^1(X,\Omega^1_X)$. Since $\lambda$ is a degree zero operator, it preserves the gradings, and so for every degree $k$ we have a long exact sequence
\[ \ldots \rightarrow  H^1(X, \of_X(k)) \xrightarrow{\lambda} H^1({U_X}, \T_{U_X})_k \rightarrow  H^1(X, \Theta_X(k)) \rightarrow H^2(X, \of_X(k)) \rightarrow \ldots \]
To conclude we have to identify $\Lambda$  with the class of an hyperplane. Again, we refer to \cite{atiyah}, where it is shown that, for a general line bundle $L$, the extension class $[0 \to \Omega^1_X \to \mathcal{L}_0 \to \of_X\to 0]$ is the first Chern class $c_1(L)$. So, for $L= \mathcal{O}_X(1)$ we find that the extension class is $c_1(\of_X(1))$, as desired. 
\end{proof}

\begin{cor}\label{cor.first}
Let $X$ be a smooth subcanonical projectively normal variety of dimension $n$, and let $m\in \mathbb{Z}$ be the integer such that $\omega_X \cong \of_X(m)$. Then the relative tangent sheaf exact sequence
\[
0 \rightarrow \Theta_{U_X/X} \rightarrow \Theta_{U_X} \stackrel{\dd\pi}{\rightarrow} \pi^*(\Theta_X) \rightarrow 0
\]
induces a long exact sequence 
\[
\ldots \rightarrow H^{n-1,0}(X) \stackrel{\lambda}{\rightarrow} H^{n,1}(X) \rightarrow H^1({U_X}, \T_{U_X})_{m} \rightarrow H^{n-1,1}(X) \stackrel{\lambda}{\rightarrow} H^{n,2}(X) \rightarrow \ldots
\]
where the maps $\lambda$ are the Lefschetz operators.
\end{cor}
\begin{proof}
Since $\omega_X \cong \of_X(m)$, Serre duality gives canonical isomorphisms \[H^i(X, \Theta_X(m))\simeq H^{n-1,i}(X)\] and \[H^i(X, \of_X(m))\simeq H^{n,i}(X).\] The result then follows from Lemma \ref{lemma.first} for $k=m$.
\end{proof}
We are now ready to prove the main result of this section.
\begin{thm}\label{t1} Let $X$ be a smooth subcanonical projectively normal variety of dimension $n$, and let $m\in \mathbb{Z}$ be the integer such that $\omega_X \cong \of_X(m)$. There is a natural isomorphism 
\[H^1(U, \Theta_{U_X})_{m} \cong H^{n-1,1}_{\mathrm{prim}}(X).\]
\end{thm}
\begin{proof}
%
%We subdivide the proof in two cases: first we consider the (easier) case in which $h^{n-1,0}(X)=0$, then we remove this assumption and generalise the result.
%\begin{case}
%Assume now that $h^{n-1,0}=0$: thanks to this, the first term is zero; moreover, recalling the pairing \[\Omega^1_X \otimes \Omega^{n-1}_X \rightarrow \of_X(m) \] and using Serre Duality on $H^2(X, \of_X(m))$, the above sequence translates into \[ 0 \rightarrow (H^1({U_X}, \T_{U_X}))_{m} \stackrel{\varphi}{\rightarrow}H^{n-1,1}(X) \stackrel{\psi}{\rightarrow} H^{n,2}(X) \rightarrow \ldots \]
%But, since we proved before that the map $\psi$ can be identified with the Lefschetz operator, and since thanks to the vanishing the map $\varphi$ is injective, we can identify $(H^1({U_X}, \T_{U_X}))_{m}$ with its image under $\varphi$. But, since the previous sequence is exact, we have that $\im (\varphi)= \Ker(\psi)$, that is $(H^1({U_X}, \T_{U_X}))_{m}\cong H^{n-1,1}_{prim}$, and we have done.
%
%\end{case}
%\begin{case} Let us remove the assumption of the vanishing of $h^{n-1,0}(X)$. 
%Again, we are left with the sequence 
Consider the long exact sequence
\[\ldots \rightarrow H^{n-1,0}(X) \stackrel{\lambda_{n-1,0}}{\rightarrow} H^{n,1}(X) \rightarrow H^1({U_X}, \T_{U_X})_{m} \rightarrow H^{n-1,1}(X) \stackrel{\lambda_{n-1,1}}{\rightarrow} H^{n,2}(X) \rightarrow \ldots\]
from Corollary \ref{cor.first}. It induces the short exact sequence
\[
0{\rightarrow} \Cok(\lambda_{n-1,0}) \rightarrow H^1({U_X}, \T_{U_X})_{m} \rightarrow \Ker(\lambda_{n-1,1})\to 0.\]
Now notice that, by definition, $ \Ker(\lambda_{n-1,1})=H^{n-1,1}(X)_{\mathrm{prim}}$, while $\Cok(\lambda_{n-1,0})$ is zero since, $\lambda_{n-1,0}$ is an isomorphism by Hard Lefschetz.
%
%We will denote by $\lambda_{i,j}^k$ the $k$-th iteration of the Lefschetz operator applied to the subspace $H^{i,j}(X)$.\\
%We recall (as for example in \cite{carlson}) that the \emph{primitive cohomology} $H^{p,q}_{\prim}(X)$ is defined either as \[ \Ker(\lambda_{p,q}^{n-k+1}:H^{p,q}(X) \to H^{n-q+1,n-p+1}(X)), \ \textrm{where }p+q=k \]
%or equivalently \[ \Ker(\Lambda_{p,q}: H^{p,q}(X) \to H^{p-1,q-1}(X)), \]
%where $\Lambda$ is the adjoint of the Lefschetz operator, defined by the means of the Hodge star operator; in particular we have $\Ker(\Lambda_{p+1,q})=\Cok(\lambda_{p, q-1})$. Conversely, we define the \emph{coprimitive cohomology} as \[ H^{p,q}_{\coprim}(X):=\Ker(\lambda_{p,q}: H^{p,q}(X) \to H^{p+1,q+1}(X)) \] and it is immediate to see that primitive and coprimitive cohomology coincides only when $p+q=n$. Moreover, by an easy consequence of Hard Lefschetz, primitive subspaces are nonzero only for $p+q \leq n$, while on the converse coprimitive are nonzero only for $p+q \geq n$.\\
%By the Hard Lefschetz Theorem,  for $k=p+q <n$ the map \[\lambda^{n-k}: H^{p,q}(X) \to H^{n-q, n-p}(X) \] is an isomorphism, and so in particular
%\[ \lambda_{n-1,0}: H^{n-1,0}(X) \stackrel{\cong}{\rightarrow} H^{n,1}(X) ,\]  and this implies $H^1({U_X}, \T_{U_X})_{m} \cong\Ker(\lambda_{n-1,1}) \cong H^{n-1,1}_{\prim}(X)$. 
%\end{case}
\end{proof}

What we have to do now is connect the previous result to the $T^1_{A_X}$, the module of first-order deformations of the affine cone of $X$.

\begin{thm}\label{t2} Let $X$ be a smooth subcanonical projectively normal variety of dimension $n>1$, and let $m\in \mathbb{Z}$ be the integer such that $\omega_X \cong \of_X(m)$. If $H^1(X, \of_X(k))=0$ for every $k\in \mathbb{Z}$, then we have \[ (T^1_{A_X})_{m} \cong H^{n-1,1}_{\prim}(X) \]
\end{thm}
\begin{proof}
From \cite{schlessinger} we have that $T^1_{A_X}$ fits into the exact sequence 
 \[ 0 \rightarrow T^1_{A_X} \to H^1({U_X}, \T_{U_X}) \to H^1({U_X}, (\T_{\C^{N+1}})\bigr|_{U_X})\cong  H^1({U_X}, \of_{U_X})^{N+1}.
\]
Since $H^1({U_X}, \of_{U_X})\cong \bigoplus_{k}H^1(X, \of_X(k))$, we see that if $H^1(X, \of_X(k))=0$ for every $k\in \mathbb{Z}$, then $T^1_{A_X} \cong H^1({U_X}, \T_{U_X})$. The conclusion the follows from Theorem \ref{t1}.
\end{proof}
\begin{cor}Under the same hypothesis of the theorem above, we have that in general the degree $k$ component of the $T^1_{A_X}$ is given by \[  (T^1_{A_X})_{k} \cong \Ker(\lambda: H^1(X, \Omega^{n-1}(k-m)) \to H^2(X, \omega_X(k-m)) .\]
\end{cor}
%It is also important to notice that any Cohen-Macaulay projective variety with $H^1(X, \of_X)=0$ satisfies our property with respect to some embedding: to see this, take a sufficiently ample line bundle $L$ such that $H^1(X,\omega_X \otimes L)=0$ and $H^1(X,L^n)=0$ for all $n\geq 1$. Thank to this, we are going to assume directly this condition for $X$.
%Also, later on in this paper, we are going to ask for maximal depth (that can be translated to $H^i(X, \of_X)=0$ for all $1<i<n$), but we stress that by now the result holds in this much weaker hypothesis.

%We are now in position to prove the main result of this section.

\begin{rmk} 
What kind of varieties satisfies the condition $H^1(X, \of_X(k))=0$ for every $k\in \mathbb{Z}$ that appears in Theorem \ref{t2} above? By Kodaira vanishing 
one sees that all smooth Fano manifolds and simply connected projective Calabi-Yau manifolds satisfy this condition. Also, every arithmetically Cohen-Macaulay projective variety (and so, in particular projective spaces and their products, projective complete intersections, Grassmann manifolds and Schubert subvarieties, flag manifolds and generalized flag manifolds) of dimension at least 2 satisfies it. Notice that, if  dim $X \geq 2$, the vanishing condition condition $H^1(X, \of_X(k))=0$ for every $k\in \mathbb{Z}$ is actually equivalent to the condition   $\mathrm{depth}_0A_X \geq 3$. Namely, we can identify $H^1({U_X}, \of_{U_X})$ with $H^2_{\mathfrak{m}}(A_X, \of_{A_X})$, the second local cohomology group of $A_X$ at the maximal (irrelevant) ideal, and the vanishing of this is by definition the same request as  $\mathrm{depth}_0A_X \geq 3$.

For dim $X\geq 2$ the simplest example of projective manifolds for which $H^1(X, \of_X(k))$ does not vanish for every $k$ given by Abelian varieties: for them, theorem \ref{t1} still holds, but the problem of determining the image of $T^1_A$ inside $H^1({U_X}, \T_{U_X})$ remains open.
\end{rmk}

%\begin{rmk}
%The depth condition we asked for automatically excludes every curve. Still, in the case of a complete intersection curve, we have $H^1({U_X}, \Theta_{\C^N}|_{U_X})=0$ (see \cite{depth}). This means we can write down an exact sequence as before
%\[ \ldots \rightarrow H^0(X, \Theta_X(m) ) \rightarrow H^1(X, \of_X(m)) \rightarrow (T^1_A)_{m} \rightarrow H^1(X, \Theta_X(m)) \rightarrow H^2(X, \of_X(m)) \rightarrow \ldots\]
%Now, since for a smooth curve $\Omega^1_X \cong \omega_X \cong \of_X(m)$, we can twist back to obtain \[ \of_X \cong (\Omega_X^1(-m) \cong (\Theta_X(m))^{\vee} \] which gives us 
%\[ \ldots \rightarrow H^0(X, \of_X) \stackrel{\gamma}{\longrightarrow} H^1(X, \omega_X) \stackrel{\beta}{\longrightarrow} (T^1_A)_{m} \stackrel{\delta}{\longrightarrow} H^1(X, \of_X) \rightarrow 0 \]
%Now, since the sequence is exact, we have that $\Ker \delta= \im \beta$; on the other hand, since $\gamma$ is non-degenerate beween 1-dimensional spaces, then it is an isomorphism. Thus dim($\Ker \delta)=0$, and this implies \[ (T^1_A)_{m} \cong H^1(X, \of_X), \]
%as previously guessed.
%\end{rmk}

\subsection{Obstructions and Automorphisms}
Now we look at the Obstruction Theory of the cone $A_X$. %; similarly to the first-order deformation module, we have that i
Infinitesimal obstructions to deformations of $A_X$ live inside
\[ T^2_{A_X}:=\Ext^2_{\of_{A_X}}(\Omega^1_{A_X}, \of_{A_X}). \] 
Let us stick to the case of depth$_0 A_X \geq 3$ and dim $X \geq 2$, so that $H^1(X, \of_X(k))=0$ for any $k$ as in the previous section. Following \cite{schlessinger}, we can identify $T^2_{A_X}$ with $H^1({U_X}, N_{U_X})$, where $N_{U_X}$ is the normal bundle of $U_X$ in $\mathbb{C}^{N+1}$.
From the defining exact sequence \[ 0 \rightarrow \Theta_{U_X} \rightarrow \Theta_{\C^{N+1}}|_{U_X} \rightarrow N_{U_X} \rightarrow 0\] 
for  $N_{U_X}$ we obtain the long exact sequence
\begin{equation} \label{eq}\ldots \rightarrow 0 \rightarrow H^1({U_X}, N_{U_X}) \rightarrow H^2({U_X}, \Theta_{U_X}) \rightarrow \left(\!\bigoplus_{k}H^2(X, \of_X(k))\!\right)^{N+1}  \rightarrow \ldots \end{equation}
in cohomology, so that if  $H^2(X, \of_X(k))=0$ for any $k$ we have an isomorphism \[T^2_{A_X}\cong H^2({U_X}, \Theta_{U_X}).\] 
Notice that the condition $H^i(X, \of_X(k))=0$ for any $k$ in $\mathbb{Z}$ and for $i=1,2$ is in particular satisfied by every arithmetically Cohen-Macaulay variety of dimension at least 3.
From Lemma \ref{lemma.first} we have the exact sequences
\[
\ldots \rightarrow 0 \rightarrow (T^2_{A_X})_{k} \rightarrow H^2(X, \Theta_X(k)) \stackrel{\lambda}{\rightarrow} H^3(X, \of_X(k)) \rightarrow \ldots\]
where the maps $\lambda$ are the Lefschetz operators.
%
%As before, we have the sequence \[ 0 \rightarrow \of_{U_X} \rightarrow \Theta_{U_X} \rightarrow \pi^* \Theta_X \rightarrow 0 \] where $\pi$ is the projection from the vertex. If we use the notation \[ H^k_*(X, \mathcal{F}) = \bigoplus_{k \in \Z} H^k(X, \mathcal{F}(k)) \] we have the usual sequence \[ \ldots H^2_*(X, \of_X) \rightarrow (T^2_A)_* \rightarrow H^2_*(X, \Theta_X) \rightarrow H^3_*(X, \of_X) \rightarrow \ldots \]
Let us restrict to the cases $k=0$ and $k=m$, where $m$ is the integer such that $\omega_X\cong\mathcal{O}(m)$. For $k=0$ we find the exact sequence
\[
\ldots \rightarrow 0 \rightarrow (T^2_{A_X})_{0} \rightarrow H^2(X, \Theta_X) \stackrel{\lambda}{\rightarrow} H^3(X, \of_X) \rightarrow \ldots\]
which identifies $(T^2_A)_0$ with a subspace of $H^2(X, \Theta_X)$, which is the space containing the obstruction to (non-immersed) deformations of $X$. If moreover 
$\lambda\colon H^2(X, \Theta_X)\to  H^3(X, \of_X)$ is the zero map (as is the case, e.g., if $H^3(X, \of_X)$ vanishes), then we have an isomorphism  $(T^2_A)_0\cong H^2(X, \Theta_X)$.
If instead we look at the $k=m$ case, then by Corollary \ref{cor.first} we have the long exact sequence
\[
\ldots \rightarrow H^{n-1,1}(X) \stackrel{\lambda}{\rightarrow} H^{n,2}(X) \rightarrow (T^2_A)_{m} \rightarrow H^{n-1,2}(X) \stackrel{\lambda}{\rightarrow} H^{n,3}(X) \rightarrow \ldots
\]
and so we get the following
\begin{thm} 
Let $X$ be a smooth subcanonical projectively normal variety of dimension $n$, and let $m\in \mathbb{Z}$ be the integer such that $\omega_X \cong \of_X(m)$. If $H^i(X, \of_X(k))=0$ for every $k\in \mathbb{Z}$, and for $i=1,2$ then we have a natural isomorphism 
\[ (T^2_A)_{m}\cong H^{n-2,1}_{\mathrm{\prim}}(X). \]
\end{thm}
\begin{proof}
By the above discussion, we have a natural short exact sequence
\[
0\rightarrow \Cok(\lambda_{n-1,1}) \to (T^2_A)_{m} \rightarrow \Ker(\lambda_{n-1,2})\to 0.
\]
By Hard Lefschetz, $\Cok(\lambda_{n-1,1})=0$, and 
\[
\Ker(\lambda_{n-1,2})= \lambda_{n-2,1}\Ker(\lambda_{n-1,2}\lambda_{n-2,1})\cong H^{n-2,1}_{\prim}(X).
\]
%In particular, since $n+2>n$ we have $\Cok(\lambda_{n-1,1})\cong H^{n,2}_{\prim}(X)=0$ and thus, using $\omega_X=\of_X(m)$, we get $(T^2_A)_{m} \cong \Ker(\lambda_{n-1,2}) =: H^{n-1,2}_{\coprim}(X)$.
\end{proof}
%
%\begin{rmk}If we remove the hypothesis of $H^{n-2,0}(X)=0$, we can work directly on the $H^2({U_X}, \T_{U_X})$, where the relation with the $T^2_A$ is given by formula \eqref{eq}. In particular we have the exact sequence 
%\[ \ldots H^{n-1,1}(X) \stackrel{\eta}{\to} H^2({U_X}, \T_{U_X})_{m} \stackrel{\varphi}{\to} H^{n-1,2}(X) \to H^{n,3} \to \ldots\]
%We also have \[H^{n-1,2}_{\prim}(X) \cong H^2({U_X}, \Theta_{U_X})_{m}/ \Ker \varphi \cong H^2({U_X}, \Theta_{U_X})_{m}/\im \eta \] and this implies \[ H^2({U_X}, \Theta_{U_X})_{m} \cong H^{n-1,2}_{\prim}(X) \oplus H^{n-1,1}_{\coprim}(X). \]
%As we are going to see, this is a special case of a more general phenomenon that we are going to explore later.
%\end{rmk}
%
A similar result holds for the $T^0_{A_X}$, the module that parametrize the infinitesimal automorphism of the affine cone $A_X$. If depth$_0A_X \geq 2$ (which is satisfied, e.g., if $X$ is normal), we have $T^0_{A_X} \cong H^0({U_X},  \Theta_{U_X})$ and so Corollary \ref{cor.first} gives
 the long exact sequence
\[
0 \rightarrow H^{n,0}(X) \rightarrow (T^0_A)_{m} \rightarrow H^{n-1,0}(X) \stackrel{\lambda_{n-1,0}}{\rightarrow} H^{n,1}(X) \rightarrow \ldots
\]
and so the short exact one
\[0 \rightarrow H^{n,0}(X) \rightarrow (T^0_A)_{m} \rightarrow \Cok(\lambda_{n-1,0}) \to 0.\]
By Hard Lefschetz, $\lambda_{n-1,0}$ is an isomorphism and so  $\Cok(\lambda_{n-1,0})=0$, while \[H^{n,0}(X)  = H^{n,0}_{\prim}(X).\]
Thus we have
\begin{thm} \label{T^0 class}
Let $X$ of dimension n be smooth, projective, with $\omega_X \cong \of_X(m)$. Then we have an isomorphism \[(T^0_A)_{m} \cong H^{n,0}_{\prim}(X). \]
\end{thm}
\subsection{A SINGULAR appendix: how to compute Hodge numbers using the $T^i$}
One of the many applications of our theorems on the $T^1_{A_X}$ and the $T^2_{A_X}$ is a concrete tool to compute part of the Hodge structure of a smooth projective variety. We recall from the previous section that, under appropriate hypothesis on depth at the vertex, we can identify \[ (T^1_{A_X})_m \cong H^{n-1,1}_{\prim}(X) \] \[ (T^2_{A_X})_m \cong H^{n-2,1}_{\prim}(X), \] where as usual $\omega_X \cong \of_X(m)$. The key is that both $T^1$ and $T^2$ are easily computable, especially using computer algebra languages such as SINGULAR (see \cite{singular}), already endowed with efficient built-in tools. Suppose we start from $X \subset \PP^N$ a smooth projective variety, with $X=V(I)$, where $I=(f_1, \ldots, f_m)$. Then the instruction \begin{verbatim}
module T_1= T_1(I);
module T_2=T_2(I);
hilb(T_1,2);
hilb(T_2,2);
\end{verbatim}
computes the dimension of both $T^1_{A_X}[-d]$ and $T^2_{A_X}[-d]$, where $d=$ max$\lbrace$deg$(f_i)\rbrace$ and then lists the dimensions of the various graded components. Let us pursue a couple of nontrivial example in detail:
\subsubsection{A (Gushel-Mukai) Fano Threefold of Degree 10 and Coindex 1} Consider the case of $X=Gr(2,5) \cap H_1 \cap H_2 \cap Q$ a threefold complete intersection in the Grassmannian of 2-planes in $\C^5$ given by two hyperplane sections and one quadric, considered for example in \cite{mukai} and  \cite{debarre2} for its connections with HyperK\"ahler geometry. Since the canonical class of the Grassmannian is $\omega_{Gr(2,5)} \cong \of_{Gr(2,5)}(-5)$, by adjunction $X$ is a Fano of index 1, that is $\omega_X= \of_X(-1)$. By Kodaira vanishing we have $h^{i,0}(X)=0$ for $i=1,2,3$, and by Lefschetz hyperplane theorem we have $H^{1,1}(X)= H^{2,2}(X) \cong \C$. Thus the only Hodge piece missing is $H^{2,1}(X)= H^{2,1}_{\prim}(X)$, and by our theorem we have \[ H^{2,1}_{\prim}(X) \cong (T^1_{A_X})_{-1}. \]
Let us produce a code in SINGULAR: \begin{verbatim}
ring r=97, (y_0, y_1, y_2, y_3, y_4, y_5, y_6, y_7, y_8, y_9), ds;
ideal I =y_2*y_4-y_1*y_5+y_0*y_7, y_3*y_4-y_1*y_6+y_0*y_8,
y_3*y_5-y_2*y_6+y_0*y_9, y_3*y_7-y_2*y_8+y_1*y_9,
 y_6*y_7-y_5*y_8+y_4*y_9, 
sparsepoly(1,1,0,10),sparsepoly(1,1,0,10),sparsepoly(2,2,0,10);
module T= T_1(I);
hilb(T,2);
\end{verbatim}
the first 5 equations in the ideal are nothing but the Pl\"ucker relations of $Gr(2,5)$ embedded in $\PP^9$, with three generic hyperplane sections and one quadratic equation (actually SINGULAR can check if those four equations form a regular sequence). The dimension of the graded component we get in return is \begin{verbatim}
1,10,22,11,1
\end{verbatim}
where we have to look at the component of degree 1, since $ (T^1_{A_X})_{-1}=(T^1_{A_X})[-2]_1$, and this is 10. If needed, we can ask for an explicit monomial basis for $H^{2,1}(X)$, using the command \begin{verbatim}
kbase(T,1);
\end{verbatim}
Note that $(T^1_{A_X})_0$ is 22-dimensional: this agrees (and the same for the Hodge numbers) with the computation considered in \cite{debarre2}.  In particular its lower Hodge diamond is 
\[\begin{matrix}

0 & & 10& & 10& & 0 \\
&0 & & 1  & &0 & \\
&&0& &0 & &\\
& & &1& && \\
\end{matrix}\]
\subsubsection{A Pfaffian-Calabi Yau Threefold}
Consider now the case of a Pfaffian-Calabi Yau Threefold, as in the work of \cite{batyrev} and \cite{kanazawa}. We define \[P= G \cap Q_1 \cap Q_2 \cap H,\] where $G=Gr(2,5)$ as before, $Q_1$ and $Q_2$ are quadrics, $H$ is an hyperplane. By a computation analogous to the previous example it is easy to see that $P$ is a Calabi-Yau threefold: the dimension of the graded component of the $T^1$ of its affine cone are
\begin{verbatim}
2,19,61,101,82, 29, 3
\end{verbatim} 
In particular we find the lower Hodge diamond
\[\begin{matrix}

1 & & 61& & 61& & 1 \\
&0 & & 1  & &0 & \\
&&0& &0 & &\\
& & &1& && \\
\end{matrix}\]
\subsubsection{A weighted example}
As said before, the method we implemented works well also in the weighted projective space case. As an example, we look at the online database \cite{grdb} where several thousands of families of quasi-smooth Fano threefolds are listed. In SINGULAR, we can deal with weighted projective space by specifying a weighted order on the monomial, namely using the command \begin{verbatim}
wp(a_0,...,a_n)
\end{verbatim}
where the $a_i$ are the chosen weights. As an example, we pick $X_{6,7} \subset \PP(1,1,2,2,3,5)$, corresponding to the entry 5839 in the database  \cite{grdb}. This is a codimension 2 Fano threefold of index 1, degree $7/10$ and with a $3 \times \frac{1}{2} (1,1,1), \frac{1}{5}(1,2,3)$ as Basket. Computing the $T^1$ in the same exact way as before we get $(T^1_A)_{-1}=39$, and we can then draw the lower Hodge diamond as
\[\begin{matrix}

0 & & 39& & 39& & 0 \\
&0 & & 1  & &0 & \\
&&0& &0 & &\\
& & &1& && \\
\end{matrix}\]

\section{Deformations of Derived Categories and Hodge Theory}
\subsection{A Primer on Noncommutative Schemes and Hochschild Structures} 
Remark \ref{derived} suggests that one natural way to generalise the notion of deformation of a $\mathbb C$-scheme $Y$ is by considering derived deformations. Namely, derived deformations of $Y$ are already encoded in the cotangent complex. On the other hand another interesting generalisation consists in deforming $Y$ as a noncommutative schemes -- whatever these structures could be. \\
The theory of noncommutative schemes, usually known as Noncommutative Algebraic Geometry, is very much a developing subject, whose fundamentals have not been completely settled yet; however the basic idea -- which dates back to Grothendieck and has recently gone through a rapid development -- consists of observing that the geometry of $Y$ does not really depend on the scheme as a space, but rather on the derived category $D(\mathfrak{Coh}\left(Y\right))$ 
%\footnote{Sometimes we may need to work with $D(\mathfrak{QCoh}\left(Y\right))$ rather than $D(\mathfrak{Coh}\left(Y\right))$.} 
or better on its dg-enhancements or, more generally $A_\infty$-enanchments. Therefore, a (non-necessarily commutative) scheme over $\mathbb C$ can be thought as the datum of an $A_\infty$-category over $\mathbb C$: those which are quasi-equivalent to a dg-category of coherent sheaves over a classical scheme will encode the usual commutative schemes, whereas the others will be called noncommutative schemes. In particular, a  $A_\infty$- deformation of the commutative scheme $Y$ will be a deformation of the dg-category category $D(\mathfrak{Coh}\left(Y\right))$ as an $A_\infty$-category \cite{Ainf}. \\
Exactly as the infinitesimal deformation theory of an associative algebra (or more generally an $A_\infty$-algebra) is governed by its Hochschild cohomology, so happens for the infinitesimal deformation theory of $A_\infty$-categories. In particular, the $A_\infty$-deformations of a commutative $\mathbb{C}$-scheme $Y$ are governed by
\[
\HH^\bullet\left(Y\right):=\HH^\bullet\left(D\mathfrak{Coh}(Y)\right)\cong\mathrm{Ext}^\bullet_{Y\times Y}\left(\of_{\Delta},\of_{\Delta}\right)
\]
where $\of_{\Delta}$ stands for the structure sheaf of the diagonal in $Y\times Y$, see \cite{caldararu}. \\
There are some even more concrete interpretations of Hochschild cohomology: as a matter of fact if $Y$ is a smooth quasiprojective variety the Hochschild–-Kostant–-Rosenberg Theorem (see \cite{caldararu}) establishes an isomorphism between $\HH^{\bullet}(Y)$ and the (cohomology) algebra of polyvector fields on $Y$, namely 
\[ 
\HH^{\bullet}(Y) \cong \bigoplus_{p,q} H^q(Y, \W^p \Theta_Y). 
\]
Notice that the algebra of polyvector fields $\bigoplus_{p,q} H^q(Y, \W^p \Theta_Y)$ is known to be the tangent space at $Y$ to the \emph{extended moduli supermanifold of complex structures} (see \cite{barannikov}) and is also related to the derived moduli of non-commutative polarized schemes recently studied by Behrend and Noohi \cite{behrend}. %: this means that Hochschild cohomology has a deep modular meaning.
Finally, if $Y$ is a projective Calabi-Yau,  Serre duality gives a canonical isomorphism 
\[ 
\HH^{\bullet}(Y) \cong \HH_{\bullet}(Y)[n] 
\] 
where the Hochschild homology of $Y$ is identified the ``vertical slices'' of the Hodge diamond of $Y$:
\[
\HH_{\bullet}(Y)=\bigoplus_k \HH_k(Y)=\bigoplus_k\left( \bigoplus_{p-q=k}H^{p,q}(Y)\right).
\] 
\subsection{Hochschild Cohomology of Punctured Affine Cones}\label{sec.hochschild}
In the context of Hochschild structures, the results in Section 2.2 and Section 2.3 have beautiful generalizations. 
%; from now on let us stick to assumptions and notations of Section 2.2, that is let $X$ be a smooth complex projective variety for which $\omega_X \cong \of_X(m)$ ($m \in \Z$), $A$ be either its affine cone or the graded algebra for which $X\cong\mathrm{Proj}(A)$ and $U_X:=A-v$ be the punctured cone. Notice that $U_X$ is a quasiprojective variety and consider its Hochschild cohomology $\HH^\bullet({U_X})$: this, via the $\mathbb G_m$-action induced by the grading of $A$, has a natural $\Z$-grading. We define 
%\[
%H^{p,q}_{\coprim}(X)\colon=\Cok(L_{p,q})
%\] 
%with $L_{p,q}$ denoting the Lefschetz operator applied to $H^{p,q}(X)$; in other words we have 
%\[
%H^{p,q}_{\coprim}(X)\colon= \Ker (\Lambda_{p,q} \colon H^{p,q}(X) \to H^{p-1,q-1}(X))
%\] 
%where $\Lambda_{p,q}$ is the formal adjoint to the Lefschetz operator with respect to the Hodge form.\\
We start proving the following straightforward generalization of Lemma \ref{lemma.first}.
\begin{lem}\label{lemma.wedge}
For every $k\in \mathbb{Z}$,and for every $p\geq 0$, the relative tangent sheaf exact sequence
\[
0 \rightarrow \Theta_{U_X/X} \rightarrow \Theta_{U_X} \stackrel{\dd\pi}{\rightarrow} \pi^*(\Theta_X) \rightarrow 0
\]
induces a long exact sequence 
\[ \ldots  \rightarrow H^q(X, \W^{p-1}\!\!\T_X(k)) \rightarrow H^q({U_X}, \W^{p} \!\T_{U_X})_k \rightarrow H^q(X, \W^{p}\!\T_X(k)) \xrightarrow{\lambda} H^{q+1}(X, \W^{p-1} \!\!\T_X(k)) \rightarrow \ldots \]
where the maps $\lambda$ are the contractions with the hyperplane class in $H^1(X,\Omega^1_X)$, and where for $p=0$ one is setting $\W^{-1}\T_X=0$.
\end{lem}
\begin{proof}
The differential $\dd\pi\colon \Theta_{U_X}\to \pi^*(\Theta_X)$ induces a short exact sequence of sheaves of $\of_{U_X}$-algebras
\begin{equation}\label{eq.algebras}
0\to\ker(\dd\pi)\to \W^\bullet  \Theta_{U_X}\xrightarrow{\dd\pi}\pi^*\W^\bullet \Theta_X\to 0,
\end{equation}
which in every homogeneous degree $p$ reads
\[ 0 \rightarrow  \pi^* \bigwedge^{p-1}\Theta_X \rightarrow \W^p \Theta_{U_X} \rightarrow \pi^* \W^p \Theta_X \rightarrow 0 \]
since the leftmost term in the relative tangent sheaf exact sequence is a trivial line bundle (see, e.g., \cite[Theorem 4.1.3]{hirzebruch}).
%Recall the sequence \[ 0 \to \of_{U_X} \to  \Theta_{U_X} \stackrel{\dd \pi}{\to} \pi^* \T_X \to 0 \]
%where the map $\dd\pi$ is the tangent map to the projection. Up to a choice of representative, we can slightly abuse the notation, and call $\pi^* v_i=v_i$, if $\lbrace v_i \rbrace_{i=1}^n$ is a local basis of sections for $\Theta_X$. Then, if we take the completion of the basis for $\pi^*\Theta_X$ -- written as $v_1, \ldots, v_n, v_{n+1}$ -- as basis for $\Theta_{U_X}$, we can write $\dd\pi (v_i)=v_i$, for $i=1, \ldots, n$, $\dd \pi (v_{n+1})=0$.\\
%Now we take  as the map betweeen  \[\W^p \Theta_{U_X} \stackrel{\W ^p\dd \pi}{\rightarrow }\pi^* \W^p \Theta_X \to 0 \] the p-exterior power of $\dd \pi$, and this is surjective. By linearity, $\W^p \dd \pi (v_{i_1} \wedge \ldots \wedge v_{i_p})=0$ if and only if $i_j=n+1$ for some index. Therefore the kernel of $\W ^p\dd \pi$ can be identified with the subspace generated by $v_{n+1} \wedge v_I$, where $I$ is any multi-index of length $p-1$ missing $n+1$, and this is isomorphic to $\pi^*\W^{p-1}\T_X$ (with the first map being the inclusion). 
%In order to simplify index notations, we will use $n-p, n-p-1$.\\
Since (\ref{eq.algebras}) is a square zero extension, the connecting homomorphisms in the long exact sequence
\[ \ldots  \rightarrow H^q(U_X, \pi^*\W^{\bullet-1}\Theta_X) \rightarrow  H^q({U_X}, \W^{\bullet} \!\T_{U_X}) \rightarrow  H^q(U_X, \pi^*\W^{\bullet}\!\T_X) \xrightarrow{\lambda}   H^{q+1}(U_X,  \pi^*\W^{\bullet-1}\Theta_X) \rightarrow \ldots \]
are given by the connecting homomorphism for the degree 1 sequence
\[ \ldots  \rightarrow H^q(U_X, \of_{U_X}) \rightarrow  H^q({U_X}, \T_{U_X}) \rightarrow  H^q(U_X, \pi^*\T_X) \xrightarrow{\lambda}   H^{q+1}(U_X,  \of_{U_X}) \rightarrow \ldots \]
extended as a (graded) derivation. By Lemma \ref{lemma.first} we know that this is given by the contraction with the hyperplane class seen as a degree zero element in $H^1(U_X,\pi^*\Omega^1_X)$.  In particular, $\lambda$ will be degree preserving, and so we get, for every $p$ and every $k$ the long exact sequence
\[ \ldots  \rightarrow H^q(X, \W^{p-1}\!\!\T_X(k)) \rightarrow  H^q({U_X}, \W^{p} \!\T_{U_X})_k \rightarrow   H^q(X, \W^{p}\!\T_X(k)) \xrightarrow{\lambda}  H^{q+1}(X, \W^{p-1} \!\!\T_X(k)) \rightarrow \ldots, \]
where $\lambda$ is the contraction with the hyperplane class in $H^1(X,\Omega^1_X)$.
\end{proof}
Assuming $\omega_X \cong \of_X(m)$, the nondegenerate pairings $\Omega_X^i \otimes \Omega_X^{n-i} \rightarrow \omega_X$ indice isomorphisms $\W^{i}\T_X (m) \cong \Omega_X^{n-i}$. Under these isomorphisms, the contraction morphisms
\[
H^q(X, \W^{p}\!\T_X(m)) \xrightarrow{\lambda} H^{q+1}(X, \W^{p-1} \!\!\T_X(m))
\]
become the Lefschetz maps
\[
H^{n-p,q}(X) \xrightarrow{\lambda} H^{n-p+1,q+1}(X).
\]
Therefor we obtain the following.
\begin{cor}
Let $X$ be a smooth subcanonical projectively normal variety of dimension $n$, and let $m\in \mathbb{Z}$ be the integer such that $\omega_X \cong \of_X(m)$. Then we have a long exact sequence
%\begin{small}
%\[ \ldots \rightarrow H^q({U_X}, \W^{p} \T_{U_X})_{m} \rightarrow H^q(X, \W^{p}\T_X(m)) \rightarrow H^{q+1}(X, \W^{p-1} \T_X(m)) \rightarrow \ldots \]
%\end{small}
%
%%  \rightarrow H^q(X, \W^{n-p-1}\T_X(m)) \rightarrow H^{q-1}(X, \W^{n-p} \T_X(m))
%
%
%Thus we have 

%
\begin{equation}\label{les}
\cdots\to H^{n-p,q-1}(X) \stackrel{ \lambda_{n-p,q-1}}{\rightarrow} \to H^{n-p+1,q}(X) \rightarrow  H^q({U_X}, \W^{p}\T_{U_X})_{m} {\rightarrow}H^{n-p,q}(X) \stackrel{\lambda_{n-p,q}}{\rightarrow}H^{n-p+1,q+1}(X) \rightarrow \ldots 
\end{equation}

where $\lambda_{i,j}\colon H^{i,j}(X)\to H^{i+1,j+1}(X)$ is the Lefschetz operator.
\end{cor}

We can then prove the following result, expressing the Hochschild cohomology of the puntured cone $U_X$ in terms of the primitive cohomology of $X$.
\begin{thm} \label{Hoch decomposition}
In the above assumptions we have a canonical isomorphism
\[\HH^{p,q}(U_X)_m \cong H^{n-p+1,q}_{\mathrm{prim}}(X)\oplus H^{n-q,p}_{\mathrm{prim}}(X),
\]
where, for each value of $p,q$, at most one of the two summands on the right is nonzero. In particular,
\[\HH^{p,q}(U_X)_m \cong 
\begin{cases}
H^{n-p+1,q}_{\mathrm{prim}}(X) \textrm{ if } p >q \\ \\
H^{n-q,p}_{\mathrm{prim}}(X) \quad \textrm{ if } p\leq q.
\end{cases}
\]

\end{thm}

\begin{proof}
The long exact sequence \ref{les} induces the  short ones
\[ 0 \to \Cok(\lambda_{n-p,q-1}) \to H^q(U, \W^{p} \T_U)_m \to \Ker(\lambda_{n-p,q})\to 0. \]

If $p\leq q$, then $\Cok(\lambda_{n-p,q-1})=0$ and $\Ker(\lambda_{n-p,q})=\lambda^{-p+q} H^{n-q,p}_{\mathrm{prim}}(X)$ by Hard Lefschetz, and so 
\[
H^q(U_X, \W^{p} \T_{U_X})_m \cong H^{n-q,p}_{\mathrm{prim}}(X)
\]
in this case. Note that for $p=q$ this gives $H^p(U_X, \W^{p} \T_{U_X})_m \cong H^{n-p,p}_{\mathrm{prim}}(X)$.

If $p>q$, again by Hard Lefschetz we have $\Ker(\lambda_{n-p,q})=0$ and by definition $\Cok(\lambda_{n-p,q-1})=H^{n-p+1,q}_{\mathrm{prim}}(X)$, so that
\[
H^q(U_X, \W^{p} \T_{U_X})_m \cong H^{n-p+1,q}_{\mathrm{prim}}(X)
\]
in this case.
By setting $H^{i,j}_{\mathrm{prim}}(X)=0$ if $i+j>n$, we can summarize the above results as
\[
\HH^{p,q}(U_X)_m=H^q(U_X, \W^{p} \T_{U_X})_m \cong H^{n-p+1,q}_{\mathrm{prim}}(X)\oplus H^{n-q,p}_{\mathrm{prim}}(X)
\]
for any $p,q$.

%Now, as in the notations of \ref{t1}, we have \[ \Cok(\lambda_{p,q-1})=\Ker(\Lambda_{p+1,q})= H^{p+1,q}_{\prim}(X), \]\[\Ker(\lambda_{p,q}) =H^{p,q}_{\coprim}(X). \]
%Notice that, if $p+q=n$, $H^{p,q}_{\coprim}(X)= H^{p,q}_{\prim}(X)$ and $\Cok(\lambda_{p,q-1})=0$. On the other hand if $p+q<n$ we have $\Ker(\lambda_{p,q})=0$ by Hard Lefschetz and thus \[H^q(U, \W^{n-p} \T_{U_X})_m \cong H^{p+1,q}_{\prim}(X), \] while con the converse if $p+q <n$ then $ \Cok(\lambda_{p,q-1})=0$ and then
%\[H^q(U, \W^{n-p} \T_U)_m \cong H^{p,q}_{\coprim}(X). \]
%
%Putting everything together, we have
%
%\[\bigoplus_{p,q} H^q({U_X}, \W^{n-p}\T_{U_X})_{m} = \bigoplus_{p,q} H^{p+1,q}_{\prim}(X) \oplus H^{p,q}_{\coprim}(X) \]

%Finally we have \[ \bigoplus_{p,q} H^q({U_X}, \W^{n-p}\T_{U_X})_{m} \hookrightarrow \bigoplus_{p,q,m} H^q({U_X}, \W^{n-p}\T_{U_X})_{m}= \bigoplus_{p,q} H^q({U_X}, \W^{n-p}\T_{U_X}) =\HH^\bullet({U_X})\] and thus we proved the theorem.

\end{proof}

The above Theorem \ref{Hoch decomposition} admits a nice rephrasing in terms of the derived deformation complex of $A_X$,
\[
T^{p,q}_{A_X}:=\Ext^q_{\of_{A_X}}(\wedge^p\mathbb L_{A_X}, \of_{A_X}). 
\]
\begin{cor} Let $X$ be a smooth subcanonical projectively normal variety of dimension $n$ which is arithmetically Cohen-Macaulay, and let $m\in \mathbb{Z}$ be the integer such that $\omega_X \cong \of_X(m)$. Let $A_X$ the affine cone of $X$. Then, for every $1\leq p\leq n+1$ and $0\leq q\leq n$, we have
%
% \[ \HH^{p,q}(A)_{m} \cong H^{n-p,q+1}_{\mathrm{prim}}(X)\oplus H^{n-q,p}_{\mathrm{prim}}(X). \]
% 
 \[
(T^{p,q}_{A_X})_m= \Ext^q_{\of_{A_X}}(\Omega^p_{A_X}, \of_{A_X})_m
\cong 
\begin{cases}
H^{n-p+1,q}_{\mathrm{prim}}(X) \textrm{ if } p > q \\ \\
H^{n-q,p}_{\mathrm{prim}}(X) \quad \textrm{ if } p\leq q.
\end{cases}
\]
\end{cor}
\begin{proof}
Since $X$ is a smooth projective variety of dimension $n$, the affine cone $A_X$ is smooth in codimension $n+1$ and so for $1\leq p\leq n+1$ and $0\leq q\leq n$ we have $T^{p,q}_{A_X}=\Ext^q_{\of_{A_X}}(\Omega^p_{A_X}, \of_{A_X})$, see, e.g., \cite[Lemma 3.2]{filip}.  Since $X$ is also arithmetically Cohen-Macaulay, we have $\mathrm{depth}_0A_X \geq n$ and this, following SGA 2 Expos\`e VI, \cite{SGA2}, implies that the inclusion $U_X\hookrightarrow A_X$ induces an isomorphism $\Ext^q_{\of_{A_X}}(\Omega^p_{A_X}, \of_{A_X})\cong \Ext_{\of_{U_X}}^q(\Omega^p_{U_X}, \of_{U_X})$. Finally, since $U_X$ is smooth, we have $\Ext_{\of_{U_X}}^q(\Omega^p_{U_X}, \of_{U_X})\cong \HH^{p,q}(U_X)$.
\end{proof}
Although the above corollary is essentially a rephrasing of Theorem \ref{Hoch decomposition}, it is important to stress that, when we consider the whole affine cone, the Ext modules becomes easy to compute using computer algebra software as SINGULAR or MACAULAY2 (\cite{singular}, \cite{M2}). In particular it should be possible to write down a computer package - similar to the ne already existing for $T^1$ and $T^2$ - able to compute all of the Hodge numbers of a smooth projective arithmetically Cohen-Macaulay variety.

%\begin{cor}From the proof of the previous theorem we have that for every $k=p+q$ either the primitive or coprimitive part is zero. In particular we have \[H^q({U_X}, \W^{n-p}\T_{U_X})_{m} \cong  H^{p,q}_{\prim}(X)  \textrm{ if }p+q=n,\]
%\[H^q({U_X}, \W^{n-p}\T_{U_X})_{m}  \cong H^{p+1,q}_{\prim}(X) \textrm{ if }p+q<n, \]
%\[H^q({U_X}, \W^{n-p}\T_{U_X})_{m} \cong  H^{p,q}_{\coprim}(X)  \textrm{ if }p+q>n,\]
%
%and at the Hochshild level  the decomposition is more complicate to write in an elegant way. Nevertheless we can write\[
%\HH^k({U_X})_{m} \cong \bigoplus_{p+q=k} \delta_{p>q}H^{n-p+1,q}_{\prim}(X) \oplus  \delta_{p \leq q}  H^{n-p,q}_{\coprim}(X),\] where $\delta_{p>q}$ is zero for $p>q$, one otherwise (and similar for $\delta_{p \leq q}$).
%
%\end{cor}

%\begin{cor}If we sum all over indexes, theorem \ref{Hoch decomposition} gives us the following Hochshild-type decomposition:
%\[ 
%\HH^\bullet({U_X})_{m} \cong \bigoplus_{p,q} H^{p+1,q}_{\prim}(X) \oplus H^{n-q,n-p}_{\mathrm{prim}}(X),
%\]
%
%that in any degree $k$ takes the form of
%\[
%\HH^k({U_X})_{m} \cong \bigoplus_{p+q=k} H^{n-p+1,q}_{\prim}(X) \oplus H^{n-q,p}_{\prim}(X).
%\]
%\end{cor}
%\begin{rmk} Let us stress that in the previous decomposition \[\HH^{p,q}(U_X)_m \cong H^{n-p,q+1}_{\mathrm{prim}}(X)\oplus H^{n-q,p}_{\mathrm{prim}}(X).
%\] exactly one between the two terms on the right hand side is nonzero. So, in particular we have either \[\HH^{p,q}(U_X)_m \cong H^{n-p,q+1}_{\mathrm{prim}}(X) \textrm{ if } p \leq q \] or \[\HH^{p,q}(U_X)_m \cong H^{n-q,p}_{\mathrm{prim}}(X) \textrm{ if } p\geq q+1.\]
%\end{rmk}

\subsection{The case of a hypersurface}

The results of the previous section lead to an interesting corollary in the case of a hypersurface: we can use them to recover Griffiths' isomorphism between the  primitive cohomology of a degree $d$ smooth hypersurface $X\subseteq \PP^{n+1}$ and a distinguished graded component of the Milnor algebra of a polynomial defining $X$. We start with some preliminary Lemmas. Most of the proofs are a straightforward consequence of Lemma \ref{lemma.wedge}, of the short exact sequences
 \[ 0 \to \W^{p-i} \T_X \to \W^{p-i} \T_{\PP^{n+1}}|_X \to \W^{p-i-1}\T_X(d) \to 0 \]
 and
 \[ 0 \to \Omega^{k}_{\PP^{n+1}}(-d) \to \Omega^{k}_{\PP^{n+1}} \to \Omega^{k}_{\PP^{n+1}}|_X \to 0, \]
 of the duality isomorphisms \[\W^i\T_X(d-n-2)\cong \Omega_X^{n-i}\] and \[\W^i\T_{\PP^{n+1}}|_X(-n-2)\cong \Omega_{\PP^{n+1}}^{n+1-i}|_X,\] and of the Kodaira and Bott vanishing theorems \cite{bott} and are therefore omitted. The proof of Lemma \ref{lemma.detailed} is a bit more subtle, so it is spelled out in full detail. It also serves as an exemplification of the technique used to prove all the other Lemmas in this section. 
\begin{lem}\label{cor.questo.qui}
For $X$  a smooth, projective hypersurface in $\PP^{n+1}$ of degree $d$, we have a natural isomorphism
\[
H^{p-i}(U_X, \W^{p-i}\T_{U_X})_{m+kd}\cong H^{p-i}(X,\W^{p-i}\T_X(m+kd)),
\]
where $m=d-n-2$, for every $0\leq i\leq p\leq n$ and every $k\geq 1$. In particular,
for every $p\geq 2$ we have natural isomorphisms
\[
H^{1}(U_X,\T_{U_X})_{m+(p-1)d}\cong H^{1}(X,\T_X(m+(p-1)d))
\]
and
\[
H^{p-1}(U_X, \W^{p-1}\T_{U_X})_{m+d}\cong H^{p-1}(X,\W^{p-1}\T_X(m+d))
\]

\end{lem}
%\begin{proof} Using the duality isomorphisms $\W^i\T_X(m)\cong \Omega_X^{n-i}$, by Lemma \ref{lemma.wedge} we have a short exact sequence
%\[
%0\to \mathrm{coker}\{ \lambda_{n-p+i,p-i-1}(kd)\} \rightarrow H^{p-i}({U_X}, \W^{p-i}\T_{U_X})_{m+kd} {\rightarrow}\ker\{\lambda_{n-p+i,p-i}(kd)\}\rightarrow 0
%\]
%where $\lambda_{i,j}(kd)$ is the Lefschetz morphism $H^j(X,\Omega_X^i(kd))\to H^{j+1}(X,\Omega_X^{i+1}(kd))$. Since $H^{p-i}(X,\Omega_X^{n-p+i+1}(kd))=0$ and $H^{p-i+1}(X,\Omega_X^{n-p+i+1}(kd))=0$ by Kodaira vanishing, this reduces to
% \[
%0\rightarrow H^{p-i}({U_X}, \W^{p-i}\T_{U_X})_{m+kd} {\rightarrow}H^{p-i}(X,\Omega^{n-p+i}(kd))\rightarrow 0
%\]
%i.e., to
%\[
% H^{p-i}({U_X}, \W^{p-i}\T_{U_X})_{m+kd} \cong H^{p-i}(X,\W^{p-i}\T_X(m+kd))
%\]
%\end{proof}
%
%\begin{cor}\label{cor.questo.qui}
%For every $p\geq 2$ we have natural isomorphisms
%\[
%H^{1}(U_X,\T_{U_X})_{m+(p-1)d}\cong H^{1}(X,\T_X(m+(p-1)d))
%\]
%and
%\[
%H^{p-1}(U_X, \W^{p-1}\T_{U_X})_{m+d}\cong H^{p-1}(X,\W^{p-1}\T_X(m+d))
%\]
%\end{cor}

\begin{lem}\label{last}
Let $X$ a smooth, projective hypersurface in $\PP^{n+1}$ of degree $d$.%, with adjunction degree $m=d-n-2$. 
Then  we have a natural isomorphism
\[
H^{p-i-1}(X,\W^{p-i-1}\T_X(m+(k+1)d)) \cong H^{p-i}(X,\W^{p-i}\T_X(m+kd)),
\]
where $m=d-n-2$, for every $0\leq i\leq p-2$, with $0\leq p\leq n$. In particular,
one has a natural isomorphism
\[
H^{1}(X,\T_X(m+(p-1)d)) \cong H^{p-1}(X,\W^{p-1}\T_X(m+d)),   
\]
for any $2\leq p\leq n$.
%
%\[ H^1(U_X, \T_{U_X})_{2d-n-2} \cong H^2(U_X, \W^2 \T_{U_X})_{d-n-2}. \]
\end{lem}

\begin{lem}\label{prima}
Let $0\leq p\leq n$. If $n\neq 2p$ then we have a natural isomorphism
\[
H^{p}(U_X, \W^{p}\T_{U_X})_{m}\cong H^{p}(X,\W^{p}\T_X(m))
\]
\end{lem}
\begin{proof}
We know from Theorem \ref{Hoch decomposition} that $H^{p}(U_X, \W^{p}\T_{U_X})_{m}\cong H^{n-p,p}_{\prim}(X)$. But for a smooth hypersurface in $\PP^{n+1}$ one has $H^{p,n-p}(X) = H^{p,n-p}_{\prim}(X)$ for any $p$ such $n\neq 2p$, due to Hard Lefschetz combined with the Lefschetz hyperplane theorem. 
\end{proof}

\begin{lem}\label{lemma.detailed}
Let $0\leq p\leq n$. If $n\neq 2p$ then we have a natural isomorphism
\[ H^{p-1}(\W^{p-1}\T_X(m+d)) \cong H^p( \W^p \T_X(m)) \]
\end{lem}
\begin{proof}
From the short exact sequence
 \[ 0 \to \W^{p} \T_X \to \W^{p} \T_{\PP^{n+1}}|_X \to \W^{p-1}\T_X(d) \to 0, \]
using the duality isomorphisms $\W^i\T_X(m)\cong \Omega_X^{n-i}$ and $\W^i\T_{\PP^{n+1}}|_X(m)\cong \Omega_{\PP^{n+1}}^{n+1-i}|_X(d)$ we get to the long exact sequence 
\[ \ldots\to H^{n-p,p-1}(X) \to H^{p-1}(X,\Omega^{n-p+1}_{\PP^{n+1}}(d)|_X) \to H^{p-1}(X,\Omega_X^{n-p+1}(d)) \to \]
\[
\to H^{n-p,p}(X) \to H^{p}(X,\Omega^{n-p+1}_{\PP^{n+1}}(d)|_X) \to 0,\]
where the last zero comes from $H^{p}(X,\Omega_X^{n-p+1}(d))=0$ by  Kodaira Vanishing.\\
Now, consider the short exact sequence (see \cite{stability})
\[ 0 \to \Omega^{n-p+1}_{\PP^{n+1}} \to \Omega^{n-p+1}_{\PP^{n+1}}(d) \to \Omega^{n-p+1}_{\PP^{n+1}}(d)|_X \to 0. \]
This induces the long exact sequence 

\[ \cdots \to H^{p}(\PP^{n+1}, \Omega^{n-p+1}_{\PP^{n+1}}(d)) \to H^{p}(\PP^{n+1},\Omega_{\PP^{n+1}}^{n-p+1}(d)|_X) \to 
\]
\[
 \to H^{p+1}(\PP^{n+1},\Omega_{\PP^{n+1}}^{n-p+1})\to H^{p+1}(\PP^{n+1}, \Omega^{n-p+1}_{\PP^{n+1}}(d))\to\cdots  \]
By Kodaira vanishing we have  $H^{p+1}(\PP^{n+1}, \Omega^{n-p+1}_{\PP^{n+1}}(d))=0$.
%, and so the above long exact sequence reduces to
%\[ \cdots \to H^{p-i}(\PP^{n+1}, \Omega^{n-p+i+1}_{\PP^{n+1}}((k+1)d)) \to H^{p-i}(\PP^{n+1},\Omega_{\PP^{n+1}}^{n-p+i+1}((k+1)d)|_X) \to
%\]
%\[
%\to H^{p-i+1}(\PP^{n+1},\Omega_{\PP^{n+1}}^{n-p+i+1}(kd))\to0.
%\]
Also we have the long exact sequence
\[ \cdots \to H^{p-1}(\PP^{n+1}, \Omega^{n-p+1}_{\PP^{n+1}}(d)) \to H^{p-1}(\PP^{n+1},\Omega_{\PP^{n+1}}^{n-p+1}(d)|_X) \to
\]
\[
\to H^{p}(\PP^{n+1},\Omega_{\PP^{n+1}}^{n-p+1})\to H^{p}(\PP^{n+1}, \Omega^{n-p+1}_{\PP^{n+1}}(d))\to\cdots  \]
%Let us distinguish several cases. 
%
%The first case is $0\leq i<p\leq n$ with 
By Bott vanishing also \[H^{p}(\PP^{n+1}, \Omega^{n-p+1}_{\PP^{n+1}}(d))=0\] and \[H^{p-1}(\PP^{n+1}, 
\Omega^{n-p+1}_{\PP^{n+1}}(d))=0.\]
Now we consider two subcases. If $n\neq 2p-1$, then 
\[
H^{p}(\PP^{n+1},\Omega_{\PP^{n+1}}^{n-p+1}(d)|_X) \cong H^{p+1}(\PP^{n+1},\Omega_{\PP^{n+1}}^{n-p+1})=0
\]
and
\[
H^{p-1}(\PP^{n+1},\Omega_{\PP^{n+1}}^{n-p+1}(d)|_X) \cong H^{p}(\PP^{n+1},\Omega_{\PP^{n+1}}^{n-p+1})=0
\]
(where we used that by hypothesis $n\neq 2p$). So, we find
\[
H^{p-1}(X,\Omega_X^{n-p+1}(d)) \cong H^{p}(X,\Omega_X^{n-p})
\]
i.e.,
\[
H^{p-1}(X,\W^{p-1}\T_X(m+d)) \cong H^{p}(X,\W^{p}\T_X(m))
\]
in this case. If $n=2p-1$ then we still have 
\[
H^{p}(X,\Omega_{\PP^{2p}}^{p}(d)|_X) = H^{p}(\PP^{2p},\Omega_{\PP^{2p}}^{p}(d)|_X) \cong H^{p+1}(\PP^{2p},\Omega_{\PP^{2p}}^{p})=0
\]
while 
\[
H^{p-1}(X,\Omega_{\PP^{2p}}^{p}(d)|_X) = H^{p-1}(\PP^{2p},\Omega_{\PP^{2p}}^{p}(d)|_X) \cong H^{p}(\PP^{2p},\Omega_{\PP^{2p}}^{p})\cong\mathbb{C}.
\]
and so our initial long exact sequence becomes
\[ \ldots\to H^{p-1,p-1}(X) \xrightarrow{\eta} H^{p,p}(\PP^{2p}) \to H^{p-1}(X,\Omega_X^{p}(d)) \to H^{p-1,p}(X) \to 0. \]
By precomposing the map $\eta$ with the restriction morphism \[H^{p-1,p-1}(\PP^{2p})\to H^{p-1,p-1}(X)\] one obtains the cup product with $c_1(\of_{\PP^{2p}}(d))$, which is an isomorphism from \[H^{p-1,p-1}(\PP^{2p})\to H^{p,p}(\PP^{2p}).\] Hence $\eta$ is surjective, and so 
\[
H^{p-1}(X,\Omega_X^{p}(d)) \cong H^{p-1,p}(X). 
\]
Therefore,
\[
H^{p-1}(X,\W^{p-1}\T_X(m+d)) \cong H^{p}(X,\W^{p}\T_X(m)). 
\]
in this case, too.
\end{proof}
\begin{cor}\label{pn2}
For any $2\leq p\leq n$ with $n\neq 2p$ one has a natural isomorphism
\[
H^{1}(U_X,\T_{U_X})_{m+(p-1)d} \cong H^{p}(U_X, \W^{p}\T_{U_X})_{m}   
\]
\end{cor}
We are now left with considering the $n=2p$ case.
\begin{lem}\label{sopra}
Assume $n=2p$. Then we have a natural short exact sequence
\[ 0 \to H^{p-1}(X,\W^{p-1}\T_X(m+d)) 
\to H^{p,p}(X) \to \mathbb{C} \to 0,\]
\end{lem}
\begin{proof}
Reasoning as in the proof of \ref{prima} we get the long exact sequence\begin{small}
\[ \ldots\to H^{p-1}(X,\Omega^{p+1}_{\PP^{2p+1}}(d)|_X) \to H^{p-1}(X,\Omega_X^{p+1}(d)) 
\to H^{p,p}(X) \to H^{p}(X,\Omega^{p+1}_{\PP^{n+1}}(d)|_X) \to 0,\]\end{small}
and we have
\[
H^{p-1}(\PP^{2p+1},\Omega_{\PP^{2p+1}}^{p+1}(d)|_X) \cong H^{p}(\PP^{2p+1},\Omega_{\PP^{2p+1}}^{p+1})=0
\]
and
\[
H^{p}(\PP^{2p+1},\Omega_{\PP^{2p+1}}^{p+1}(d)|_X) \cong H^{p+1}(\PP^{2p+1},\Omega_{\PP^{2p+1}}^{p+1})\cong\mathbb{C}.
\]
\end{proof}
\begin{cor}\label{p2}
Assume $n=2p$, with $p\geq 2$. Then there exists an isomorphsim
\[
H^{1}(U_X,\T_{U_X})_{m+(p-1)d} \cong H^{p}(U_X, \W^{p}\T_{U_X})_{m}  
\]
\end{cor}
\begin{proof}
By Lemmas \ref{cor.questo.qui} and \ref{last} we have an isomorphism 
\[H^{1}(U_X,\T_{U_X})_{m+(p-1)d} \cong\break  H^{p-1}(U_X, \W^{p-1}\T_{U_X})_{m+d},\]
so we need only to exhibit an isomorphism \[H^{p-1}(U_X, \W^{p-1}\T_{U_X})_{m+d}\cong H^{p}(U_X, \W^{p}\T_{U_X})_{m}.\] To do this, recall the isomorphism $H^{p}(U_X, \W^{p}\T_{U_X})_{m}\cong H^{p,p}_{\prim}(X)$ from \ref{Hoch decomposition}, the short exact sequence
\[
0\to H^{p,p}_{\prim}(X)\to H^{p,p}(X)\xrightarrow{\lambda}  H^{p+1,p+1}(X)\to 0
\]
coming from Hard Lefschetz, and the fact that $H^{p+1,p+1}(X)\cong\mathbb{C}$ from the Lefschetz hyperplane theorem combined with Hard Lefschetz. Therefore we have a natural short exact sequence
\[
0\to H^{p}(U_X, \W^{p}\T_{U_X})_{m}\to H^{p,p}(X)\to  \mathbb{C}\to 0,
\]
and we use Lemma \ref{sopra} and Lemma \ref{cor.questo.qui} to conclude.
\end{proof}
Putting all the above results toghether we obtain the following
\begin{thm}
Let $X\subseteq \PP^{n+1}$ be a smooth degree $d$ projective hypersurface, with $\dim X = n\geq 3$. Then we have
\[
H^{1}(U_X,\T_{U_X})_{pd-n-2} \cong H^{p}(U_X, \W^{p}\T_{U_X})_{d-n-2}  
\]
for every $0\leq p\leq \dim X$.
\end{thm}
\begin{proof}
Since $d-n-2$ is precisely the integer $m$ such that $\omega_X\cong \of_X(m)$ by adjunction, for $p\geq 2$ the result follows from Corollary \ref{p2} and Corollary \ref{pn2}. For $p=1$ there is nothing to prove. Finally, for $p=0$ we have to show that
\[
H^{1}(U_X,\T_{U_X})_{-n-2} \cong H^{0}(U_X, \of_{U_X})_{d-n-2}  
\]
On the right hand side we have $H^0(X,\of_X(d-n-2))$, while on the left hand side
we consider the short exact sequence
\[
0\to \mathrm{coker}\{ \lambda_{n-1,0}(-d)\} \rightarrow H^{1}({U_X}, \T_{U_X})_{-n-2} {\rightarrow}\ker\{\lambda_{n-1,1}(-d)\}\rightarrow 0
\]
where $\lambda_{i,j}(-d)$ is the Lefschetz morphism \[H^j(X,\Omega_X^i(-d))\to H^{j+1}(X,\Omega_X^{i+1}(-d)).\] Now we have \[H^{1}(X,\Omega_X^{n}(-d))=
H^1(X,\of(-n-2))=0\] and \[H^{2}(X,\Omega_X^{n}(-d))=H^{2}(X,\of_X(-n-2))=0,\] since $X$ is arithmetically Cohen-Macaulay.  So the above short exact sequence gives
\[
 H^{1}({U_X}, \T_{U_X})_{-n-2} \cong H^{1}(X,\T_X(-n-2))
\]
and to conclude the proof of the theorem we only need to show that 
\[
H^{1}(X,\T_X(-n-2))\cong H^0(X,\of_X(d-n-2)).
\]
This follows from the normal sheaf exact sequence as in the proof of Lemma \ref{last}. 
\end{proof}

Now, why is this interesting? Thanks to the computation above, we have \[H^{p, n-p}_{\prim}(X) \cong H^p({U_X}, \W^p \T_{U_X})_{m} \cong H^1({U_X}, \T_{U_X})_{pd-n-2}.\] On the other hand we know that for a degree $d$ smooth projective hypersurface defined by the polynomial $f$ we have 
\[
H^1(U_X, \T_{U_X}) \cong T^1_{A_X}\cong \mathcal{M}_f[d].
\]
Therefore, we recover the following result from Griffiths' residue theory \cite{griffiths}.
\begin{cor}
Let $X\subseteq \PP^{n+1}$ be a degree $d$ smooth projective hypersurface  with $\dim X \geq 3$, defined by the polynomial $f$.
Then we have 
\[
(\mathcal{M}_f[d])_{pd-n-2}\cong   H^{p, n-p}_{\prim}(X)
\]
for every $0\leq p\leq \dim X$. Equivalently, one has
\[
(\mathcal{M}_f)_{pd-n-2}\cong   H^{p-1, n-p+1}_{\prim}(X)
\]
for every $1\leq p\leq \dim X+1$.
\end{cor} 

\pagestyle{plain}
\bibliographystyle{abbrv}

%\bibliography{biblio}
\end{document}